\documentclass[12pt]{amsart}
\usepackage{amssymb}
\usepackage{amsfonts}
\usepackage{amsmath}%
\usepackage{graphicx}
\usepackage[export]{adjustbox}
\usepackage{hyperref}
\usepackage{fancyhdr}
\usepackage{xcolor}
\usepackage{tikz-cd}
\makeatletter
\providecommand{\U}[1]{\protect\rule{.1in}{.1in}}
\newtheorem{theorem}{Theorem}

\newtheorem{definition}{Definition}

\newtheorem{lemma}{Lemma}
\newtheorem{notation}{Notation}

\newtheorem{proposition}{Proposition}
\newtheorem{remark}{Remark}

\begin{document}
\date{\today}
\title{A Lyapunov function for a  Synchronisation diffeomorphism  of three clocks}
\author{Jorge Buescu$^1$, Emma D'Aniello$^2$ and Henrique M. Oliveira$^{3*}$}
\thanks{$^1$jsbuescu@ciencias.ulisboa.pt; 
ORCID: 0000-0001-5444-5089; Departamento de Matemática, Faculdade de Ciências, CEMS.UL - centre for Mathematical Studies, 
FCT  UID/04561/2025,  Universidade de Lisboa, Campo Grande, 1749-006 Lisbon, Portugal.}
\thanks{$^2$emma.daniello@unicampania.it; ORCID: 0000-0001-5872-0869; Dipartimento di Matematica e Fisica, Universit\`{a} degli Studi della Campania
\textquotedblleft Luigi Vanvitelli\textquotedblright, Viale Lincoln n. 5 -
81100 Caserta, Italia}
\thanks{$^3$henrique.m.oliveira@tecnico.ulisboa.pt; ORCID: 0000-0002-3346-4915; 
Centro de An\'alise Matem\'atica, Geometria e Sistemas Din\^amicos, Departamento de Matem\'atica, Instituto Superior T\'ecnico, Universidade de Lisboa, Av. Rovisco Pais 1, 1049-001 Lisbon, Portugal;\\$^*$ Corresponding author;  }

\begin{abstract}
Lyapunov functions are essential tools in dynamical systems, as they allow the stability analysis of equilibrium points without the need to explicitly solve the system’s equations. Despite their importance, no systematic method exists for constructing Lyapunov functions. In a previous paper, we examined a diffeomorphism arising from the problem of Huygens Synchronisation for three identical limit cycle clocks arranged in a line, proving that the system possesses a unique asymptotically stable  fixed point on the torus $\mathbb{T}^2$, corresponding to synchronisation in phase opposition. In this paper, we rederive this result by constructing a discrete Lyapunov function for the system.

The closure of the basin of attraction of the  asymptotically stable attractor is the torus $\mathbb{T}^2$, showing that Huygens Synchronisation exhibits generic and robust behaviour, occurring with probability one with respect to initial conditions.
\end{abstract}
\maketitle

\section{Introduction and preparatory results}
\label{sec_intro}

\subsection{Motivation and Paper Organization}

The study of Lyapunov functions plays a pivotal role in understanding the stability of dynamical systems, offering a robust framework for the analysis of equilibrium points without requiring explicit solutions to the governing equations. This paper builds on the synchronisation problem of three coupled oscillatory systems with nearest-neighbour interactions, a scenario characterised by rich dynamical properties and practical relevance. We address this problem by constructing a discrete Lyapunov function for the Synchronisation diffeomorphism of three clocks and establishing that phase opposition between clocks is the paradigm for these systems, reinforcing and extending prior stability results.

The paper is structured as follows: Section 1 introduces the foundational concepts and theoretical framework of discrete Lyapunov functions. Section 2 details the dynamical system under study, arising from the Synchronisation problem, and its key properties. Section 3 focuses on constructing the Lyapunov function and demonstrating its applicability to stability analysis. Finally, Section 4 concludes with a synthesis of the results and their implications for broader applications in dynamical systems theory.

\subsection{Discrete Lyapunov functions}
\label{subsec_discLyap}

The stability analysis of equilibrium points is a cornerstone of dynamical systems and control theory. Lyapunov's approach \cite{lyapunov1892} to stability analysis, commonly referred to as ``Lyapunov's Second Method'' \cite{Gu1983,hirsch2012differential,lasalle1976stability}, offers a systematic framework for evaluating the stability or asymptotic stability of equilibrium points. This method employs Lyapunov functions and is applicable to both continuous-time and discrete-time dynamical systems, without requiring explicit computation of the system's flow.

In this paper, we focus on discrete dynamical systems defined by the iteration of diffeomorphisms, specifically addressing discrete-time Lyapunov stability. Discrete-time Lyapunov functions are only required to be continuous, as noted by La Salle \cite{lasalle1976stability}. Consequently, discrete Lyapunov functions naturally fall within the domain of topological dynamics.

Our results about discrete Lyapunov functions are pre\-sen\-ted within this general topological dynamics framework. From Section \ref{sec_DS} onwards, we narrow our focus to the specific case of the Synchronisation diffeomorphism of $\mathbb{R}^2$ (or $\mathbb{T}^2$) for the problem of three clocks on a line with nearest neighbour interactions, which is of primary relevance to this work.

\begin{definition}
\label{TDS}
Let $X$ be a topological space and $f: X \to X$ be a continuous map. We  refer to the pair $(X, f)$
as a {\em topological dynamical system}. 
\end{definition}

We note that this should technically be referred to as a semi-dynamical system. However, as this distinction is not relevant to the discussion that follows — given that the map studied in this paper is invertible — we omit this distinction hereafter.

\begin{definition}[Discrete orbital derivative]
\label{def_orb_deriv}
Let $(X,f)$ be a topological dynamical system and $V: X \to \mathbb{R}$  be a continuous function. 
We define the {\em discrete orbital derivative of $f$\/} as the function
\[ \dot{V}(x) = V(f(x)) - V(x). \]
\end{definition}

If $(x_n)_{n \in \mathbb{N}}$ is an orbit of the dynamical system defined by $f$ (i.e. $x_{n+1}= f(x_n)$), then $\dot{V}(x_n) =  V(x_{n+1}) - V(x_{n})$, and therefore 
$\dot{V}(x) \leq 0$ means that $V$ is nonincreasing along orbits of $f$.

\begin{definition}[Discrete Lyapunov function]
\label{def_discrete_Lyap}
Let $(X,f)$ be a topological dynamical system and $V: X \to \mathbb{R}$  be a continuous function.
Suppose $S \subset X$. We say $V$ is a {\em Lyapunov function for $f$ on $S$\/} if 
$\dot{V}(x) \leq 0$ for all $x \in S$.
\end{definition}


We then
have the following discrete version of Lyapunov's Stability Theorem, whose statement may be found in 
 \cite{lasalle1976stability}. Here, the discrete versions of Lyapunov stability and asymptotic stability are defined in a standard manner, also to be found in \cite{lasalle1976stability}; see also e.g. \cite{buescu1998exotic,Buescu1995}.

\begin{theorem}[Discrete Lyapunov Stability Theorem]
\label{Lyapunov for maps}
Let $(X,f)$ be a topological dynamical system.
Let $H$ be compact and let $S$ be an open set containing $H$. Suppose that $V(x)$ is a function such that 
\begin{enumerate}
\item $V(x) \leq 0$ for $x \in H$ and $V(x) > 0$ for $x \in S \setminus H$, and
\item $V$ is a Lyapunov function for $f$ on $S$.
\end{enumerate}
Then $H$ is Lyapunov stable. If, in addition, $\dot{V}(x) <0$ on $S \setminus H$, then $H$ is asymptotically stable.
\end{theorem}

To state our next result we need to recall the notion  of topological conjugacy of dynamical systems.

\begin{definition}
\label{def_TC}
 We say that two topological dynamical systems $(X_1, f_1)$ and $(X_2, f_2)$ are {\em topologically 
conjugate\/} if there exists a homeomorphism $h: X_1 \to X_2$ such that 
\begin{equation}
\label{eqn_TC}
h \circ f_1 = f_2 \circ h.
\end{equation}
\end{definition}
When two dynamical systems $(X_1, f_1)$ and $(X_2, f_2)$ are topologically conjugate, a homeomorphism 
$h$ performing the 
conjugation, i.e. satisfying \eqref{eqn_TC}, is called a topological conjugacy between $f_1$ and $f_2$.

\begin{proposition}
\label{thm_conjugacy}
Let $(X_1, f_1)$ and $(X_2, f_2)$ be topologically conjugate. Suppose the system $(X_1, f_1)$ admits
a Lyapunov function $V_1$, and let $S_1, \, H_1$ be the associated open and compact sets 
in Theorem~\ref{Lyapunov for maps}. Then
\begin{equation}
V_2 = V_1 \circ h^{-1}
\label{eq_Lyap_TC}
\end{equation}
is a Lyapunov function for $f_2$ associated to the sets $S_2= h(S_1)$  and $H_ 2 = h(H_1)$.
\end{proposition}

\proof 
Suppose $x_2 \in X_2$ and let $x_1 = h^{-1}(x_2)$. Then it follows  that 
\[ \dot{V_2}(x_2) = V_1(f_1(h^{-1}(x_2))) - V_1(h^{-1}(x_2)) = V_1(f_1(x_1)) - V_1(x_1))
\]
so $V_2$ is a Lyapunov function for the topological dynamical system $(X_2, f_2)$. Moreover, consideration 
of the commutative diagram
\begin{center}
\begin{tikzcd}
    X_1 \arrow[r, "f_1"] \arrow[d, "h"'] & X_1 \arrow[d, "h"] \\
    X_2 \arrow[r, "f_2"'] & X_2
\end{tikzcd}
\end{center}
shows that, if the sets $S_1$ and $H_1$ have the properties stated 
in Theorem~\ref{Lyapunov for maps} for system
$(X_1, f_1)$, then
 the corresponding sets $S_2= h(S_1)$  and $H_ 2 = h(H_1)$ have those
 properties for system $(X_2, f_2)$.
 \qed

\vspace{2mm}
Observe that in any topological space $X$ a finite set is always compact, so Theorem~\ref{Lyapunov for maps} 
and Proposition~\ref{thm_conjugacy} are, in particular, immediately applicable to the study of stability of fixed points. 
Indeed, if  $x^*$ is a fixed point of $f$, then $H= \{ x^* \}$ is compact and the Lyapunov method applies.

In this paper we deal specifically with a diffeomorphism that may be studied equivalently in $\mathbb{R}^2$ or $\mathbb{T}^2$, and our compact
positively invariant set $H$ of interest will be, precisely, a fixed point.
A recent use of Lyapunov functions in the context of discrete maps is discussed in \cite{Baigent2023}.

\section{A dynamical system arising from Synchronisation}
\label{sec_DS}

\subsection{Identical clocks} In a series of recent papers \cite{BAO2024b,EH2}, the  authors investigated the Synchronisation of three plane oscillators with an asymptotically stable limit cycle under the mechanism of Huygens Synchronisation of the second kind, that is, where the interaction is performed not via momentum transfer but by a perturbative mechanism. The model incorporated the Andronov \cite{And} pendulum clock, as used in \cite{OlMe}, but the method also applies to other types of oscillators with coupling given by the discrete Adler  equation \cite{adler1946study,Pit}. 
The theory relies only each individual system  systems having its own limit cycle, while different oscillators 
interact weakly once per cycle. This framework ensures 
applicability of the results irrespective of specific details of the oscillator models. We refer to these oscillators as clocks since we assume each oscillator to be isochronous when isolated from perturbations. 

The case of all-to-all interaction between three clocks was treated in \cite{EH2}. 
In \cite{BAO2024b}, we explore the behaviour of a  linear arrangement of three oscillators with nearest neighbour interactions; see Fig.~\ref{fig: 3_in_Line}.
\begin{figure}[tbh]
\centering
\includegraphics[height=1.3in,width=3.513in]{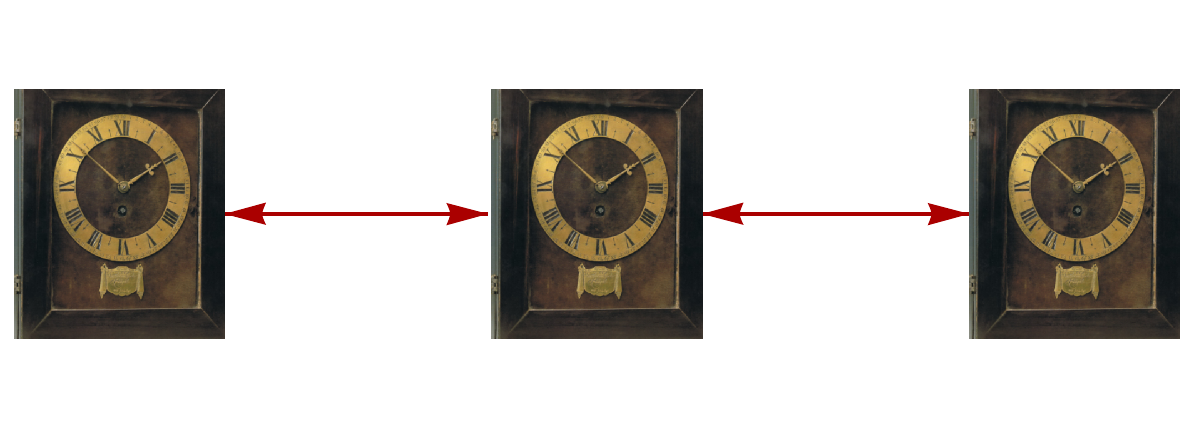}\caption{Three clocks on a line with nearest neighbour interaction.}%
\label{fig: 3_in_Line}%
\end{figure}

Denoting the two independent phase differences between oscillators $B$ and $C$ relative to the central clock $A$ by $x$ and $y$, respectively, it follows from the theory developed in \cite{BAO2024b,EH2} that the dynamics is now described by the discrete system

\begin{equation}%
\begin{bmatrix}
x_{n+1}\\
y_{n+1}%
\end{bmatrix}
=F%
\begin{bmatrix}
x_{n}\\
y_{n}%
\end{bmatrix}
=%
\begin{bmatrix}
x_{n}+2a\sin x_{n}+a\sin y_{n}\\
y_{n}+a\sin x_{n}+2a\sin y_{n}%
\end{bmatrix}
. \label{Eq3}%
\end{equation}
where $0 < a \ll 1$ is the coupling parameter. 

It is easily shown that $F$ in equation \eqref{Eq3} is a diffeomorphism for \begin{equation}\label{eq:acond}
0<a <\frac{1}{6}.
\end{equation}
This is the parameter region for $a$ in which we will work throughout the paper. 

In \cite{BAO2024b} we construct the global dynamics (see the phase portrait in Fig.~\ref{fig:06}), 
which we summarize as follows. 

\begin{figure}
[ptb]
\begin{center}
\includegraphics[
height=4.5in,
width=4in
]%
{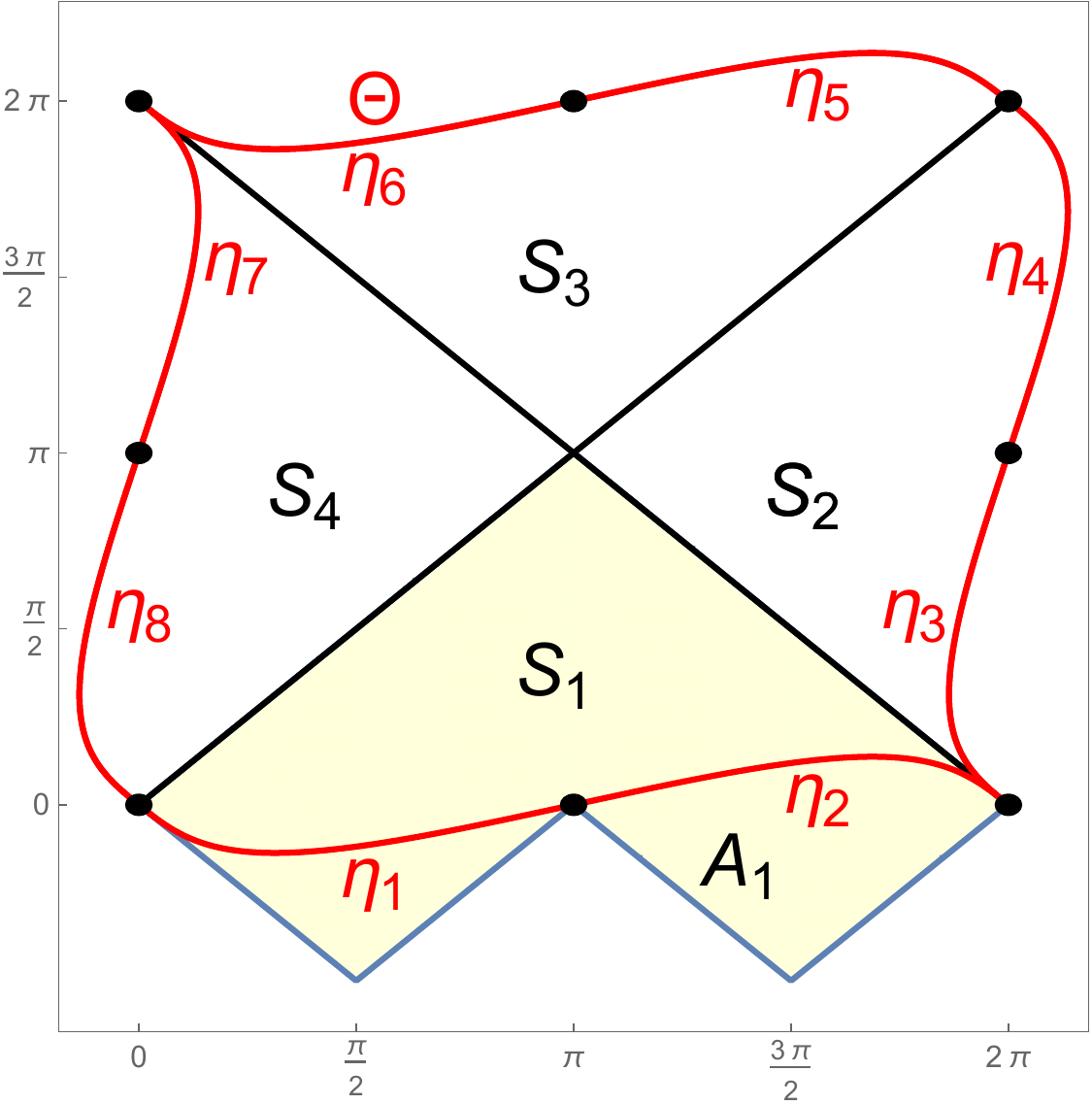}%
\caption{We show here the set $A_1$ in light yellow; the heteroclinics $\eta_j$ for $j=1,\dots,8$ connecting sources and saddles, and the set $\Theta$ in red; the sectors $S_i$ such that $S=\bigcup_{i=1}^{4} S_{i}$ are also depicted.}%
\label{Fig:05}%
\end{center}
\end{figure}

Firstly, the dynamics is periodic with period $2\pi$ in both variables.

Secondly, consider the square $D = [0, 2\pi] \times [0, 2\pi]$. In $D$, there exist 9 fixed points: unstable nodes (sources) at $(0,0)$, $(0, 2\pi)$, $(2\pi, 0)$, and $(2\pi, 2\pi)$; saddle points at $(\pi,0)$, $(0, \pi)$, $(2\pi, \pi)$, and $(\pi, 2\pi)$; and a stable node (sink) at $(\pi,\pi)$. All these fixed points are hyperbolic.

Thirdly, there exist heteroclinic connections between:
\begin{enumerate}
\item the unstable node $(0,0)$ and the saddle point at $(\pi,0)$, which we denote by $\eta_1$;
\item the unstable node $(2\pi,0)$ and the saddle point at $(\pi,0)$, which we denote by $\eta_2$;
\item the unstable node $(2\pi,0)$ and the saddle point at $(2\pi,\pi)$, which we denote by $\eta_3$;
\item the unstable node $(2\pi,2\pi)$ and the saddle point at $(2\pi,\pi)$, that we call $\eta_4$;
\item the unstable node $(2\pi,2\pi)$ and the saddle point at $(\pi,2\pi)$, which we denote by $\eta_5$;
\item the unstable node $(0,2\pi)$ and the saddle point at $(\pi,2\pi)$, which we denote by $\eta_6$;
\item the unstable node $(0,2\pi)$ and the saddle point at $(0,\pi)$, which we denote by $\eta_7$;
\item the unstable node $(0,0)$ and the saddle point at $(0,\pi)$, which we denote by $\eta_8$.
\end{enumerate}

\begin{notation} We denote by $\mathrm{Fix}^u$  the set of the $8$ local unstable fixed points (saddles and sources)  in $D$.
\end{notation}

\begin{notation} We denote by $H=\bigcup_{i=1}^{8} \eta_i$ the union of the heteroclinics defined above.
\end{notation}

\begin{notation}We denote by $\Theta$ the union of the set of  heteroclinics H and the set $\mathrm{Fix}^u$ of the unstable fixed points,
\[
\Theta=\text{H} \cup \mathrm{Fix}^u.
\]
\end{notation}

The above set $\Theta$ is an invariant set that encloses an open domain $S$, as depicted in Fig.~\ref{Fig:05}. The open set $S$ is also invariant.

Fourthly, the symmetry and periodicity properties of the diffeomorphism imply that the phase space is effectively tiled in an invariant way by translations of the domain $\bar{S}=S \cup \Theta$, 
so the dynamics may be understood as taking place on the 2-torus 
$\mathbb{T}^2= \mathbb{R}^2 / (2\pi\mathbb{Z})^2$.

The torus $\mathbb{T}^2$ is a flat topological torus obtained by identifying opposite boundaries of the set $S$. Specifically, the curves $\overline{\eta_1 \cup \eta_2}$ are identified with $\overline{\eta_5 \cup \eta_6}$, and $\overline{\eta_3 \cup \eta_4}$ are identified with $\overline{\eta_7 \cup \eta_8}$, as shown in Fig.~\ref{Fig:05}.
All metric and/or measure properties of $\mathbb{R}^2$ that hold within the invariant compact set $\bar{S}$ are preserved in $\mathbb{T}^2$ due to this construction.

In \cite{BAO2024b} we proved that $S$ is the basin of attraction of the asymptotically stable node $\left( \pi, \pi \right)$. Considering now the dynamics on the torus, we may state the next result.

\begin{proposition}
\label{cor_prob_1}
Almost all initial conditions, in the sense of Lebesgue measure,  on $\mathbb{T}^2$ approach the  Synchronised state $\left( \pi, \pi \right)$, and do so exponentially fast.
\end{proposition}

\begin{proof} The basin of the  asymptotically stable attractor $\left( \pi, \pi \right)$ is open and the only points not in the basins are the heteroclinics from source to saddle and these fixed points, i.e., the set $\Theta$, which has zero Lebesgue measure. 
Thus the basin $S$ has full measure. Exponential rates of attraction are a consequence of the hyperbolicity of the attractor.
\end{proof}

Since 
the torus has finite measure, by the appropriate normalization  this  result may be restated in terms of the corresponding probability measure: with probability $1$ every initial condition approaches the Synchronised state $\left( \pi, \pi \right)$, i.e., at phase opposition between successive clocks. 

Moreover, the structure of the diffeomorphism reveals several symmetries \cite{BAO2024b}. The lines $y=x$ and $y=2\pi-x$, which are also the support of the invariant line segments connecting the four repeller nodes to the only attracting node $(\pi,\pi)$, divide the region $S$ enclosed by $\Theta$ into four closed invariant sectors: $S_1$, $S_2$, $S_3$, and $S_4$, as depicted in Fig.~\ref{Fig:05}. 

\subsection{Equivariance}
Consider a linear bijection  $\Phi \in GL( \mathbb{T}^2)$ which commutes with $F$, that is 
\[
 F\circ\Phi \, (x,y)=\Phi\circ F \, (x,y) \ \ \ \forall (x,y) \in \mathbb{T}^2.
\]
Such $F$ is  said to be a $\Phi$-equivariant map \cite{auslander,field1970,field1980,golubitsky2015}. The set of all such $\Phi$ is easily seen to form a group under composition. This group is a linear action of the symmetry group $\Gamma$ of the map $F$; with a slight abuse of language we identify this representation with $\Gamma$ itself, so that
\[ \Gamma = \{ \Phi \in GL(\mathbb{T}^2) : F\circ\Phi \, =\Phi\circ F  \}. \]

The following proposition summarizes some standard results in equivariant dynamics of which we shall make extensive use in the last section of this article; for completeness, we state and prove it in the present context. 
Recall that a set $S$ is strongly $F$-invariant if $F(S)=S$, while an orbit with initial condition $X_0$ is the set  $O_{X_{0}}= \{ \left(
X_{n}\right) \}_{n \in \mathbb{Z}}$ such that
\[
X_{n+1}=F\left(  X_{n}\right)_, \quad n \in \mathbb{Z}.
\]

\begin{proposition}
Let $F: \mathbb{T}^2 \to \mathbb{T}^2$ be a $\Phi$-equivariant diffeomorphism. Then:
\begin{enumerate}
\item If the set $S$ is strongly $F$-invariant, then  the set $\Phi\left(  S\right)  $ is also strongly $F$-invariant.  
\item If $O_{X_{0}}$ is an orbit of $F$ with initial condition $X_{0}$, then
$
\Phi\left(  O_{X_{0}}\right)  = \{ (  \Phi (  X_{n}) \}_{n \in \mathbb{Z}}
$
is  an orbit of $F$ with initial condition $\Phi\left(  X_{0}\right)  $.
\end{enumerate}
\label{prop_invariance}
\end{proposition}

\begin{proof}
If $S$ is strongly invariant, that is $F(S) = S$, 
then $\Phi$-equivariance immediately implies 
\[
F\left(  \Phi\left(  S\right)  \right)  =\Phi\left(  F\left(  S\right)
\right)  =\Phi\left(  S\right),
\]
showing invariance of $\Phi\left(  S\right)$ and proving the first statement. 

For the second statement, notice that $\Phi$-equivariance of $F$ implies 
\[
F^{n}\circ\Phi\left(  X\right)  =\Phi\circ F^{n}\left(  X\right)   \ \ \forall n \in \mathbb{Z}
\]
and therefore, if $Y_0 = \Phi(X_0)$, then 
\[
F^n(Y_0) = F^n(\Phi(X_0)) = \Phi(F^n(X_0)) \ \ \forall n \in \mathbb{Z}, \]
finishing the proof.
\end{proof}

Consider all orbits with initial conditions on an invariant set \( S \).  
Proposition~\ref{prop_invariance} implies that any orbit in \( S \) has an equivalent orbit, in the sense of linear conjugacy, within the invariant set \( \Phi\left(S\right) \).  
More generally, the dynamics of each initial condition in \( S \) are linearly conjugate to the dynamics of the corresponding initial condition in \( \Phi\left(S\right) \).  
In other words, the flow of the dynamical system in \( S \) is linearly conjugate to the flow in \( \Phi\left(S\right) \).  

Naturally, the existence of a Lyapunov function in an open set \( S \) is equivalent to the existence of a Lyapunov function in the image of \( S \) under \( \Phi \). This result follows directly from Proposition~\ref{thm_conjugacy}, since \( \Phi \) is also a topological conjugacy.

\begin{definition}
For $(x,y) \in S$, we define the following sectors:
\begin{align*}
&S_1 = \left\{\left( x,y \right)\in S:
x \leq y \leq 2\pi - x \right\},&\\
&S_2 = \left\{\left( x,y \right)\in S: x - \pi \leq y \leq 2\pi - x \right\},&\\
&S_3 = \left\{\left( x,y \right)\in S: x - \pi \leq y \leq x \right\},&\\
&S_4 = \left\{\left( x,y \right)\in S: x \leq y \leq -x + 2\pi\right\}.
\end{align*}
\end{definition}

We next show that the dynamics in any two of these sectors  is conjugated by an involution. We 
exhibit these maps explicitly below.

\begin{definition}\label{def:phii}
Identity map $\Phi_1$:
\begin{align}\label{eq:phiID}
\Phi_1\left( x, y \right) = \left\{
\begin{aligned}
X & = x, \\
Y & = y.
\end{aligned}
\right.
\end{align}
Reflection $\Phi_2$ along the line $y = 2\pi - x$:
\begin{align}\label{eq:phi1}
\Phi_2\left( x, y \right) = \left\{
\begin{aligned}
X & = 2\pi - y, \\
Y & = 2\pi - x.
\end{aligned}
\right.
\end{align}
Rotation $\Phi_3$ of $\pi$ around the point $\left( \pi, \pi \right)$:
\begin{align}\label{eq:phi3} 
\Phi_3\left( x, y \right) = \left\{
\begin{aligned}
X & = 2\pi - x, \\
Y & = 2\pi - y.
\end{aligned}
\right.
\end{align}
Reflection $\Phi_4$ along the line $y = x$:
\begin{align}\label{eq:phi2}
\Phi_4\left( x, y \right) = \left\{
\begin{aligned}
X & = y, \\
Y & = x.
\end{aligned}
\right.
\end{align}
\end{definition}

With these definitions, we have the following relationship:
\begin{align}\label{eq:phi_relation} 
\Phi_j\left( S_1 \right) = S_{j} \text{ for } j = 1,\dots,4.
\end{align}
\begin{remark}\label{rem:simple_exercise}
It is a simple exercise to verify that \( \Phi_j = \Phi_j^{-1} \), so that \( \Phi_j \) is an involution, 
and it was proven in \cite{BAO2024b} that \( F \) is a \( \Phi_j \)-equivariant map for \( j = 1, \dots, 4 \).
\end{remark}

The structure of the dynamics in $S$ (and therefore in $\mathbb{R}^2$ or $\mathbb{T}^2$) is then completely determined by the dynamics in a particular sector, as depicted in Fig.~\ref{fig:06}. Additionally, it is sufficient to prove that the conditions of the Lyapunov Theorem hold in a single sector, as the result will follow by symmetry for the other sectors, and, consequently, to the open domain $S$, as we prove in paragraph \ref{par:full} of this article.


Before proceeding, we return to the analysis of the dynamical system in $\mathbb{R}^2$, which provides a more convenient framework for constructing the sets on which the discrete Lyapunov function will be defined. The final results will naturally extend to the torus $\mathbb{T}^2$, as the basin of attraction $S$ is an open, invariant, and proper subset of $\mathbb{T}^2$.

From the discussion of the Lyapunov function, we define the set $A_1$ were we claim the negativeness of the orbital derivative of the Lyapunov function that we construct below. Since this region $A_1$ contains the invariant set $S_1$ as shown in \cite{BAO2024b}, the orbital derivative will be negative in this last set as well. Therefore, the same happens by symmetry and $\Phi$-equivariance of $F$ in the other sectors, meaning that the node $(\pi,\pi)$ is asymptotically stable in $S$.

\begin{definition}
We define the compact set $A_1$ (see Fig.~\ref{Fig:05}) 
\begin{align*}
A_1 &= \left\{ (x,y) : 0 \leq x \leq \pi, y \leq x, y \geq -x, y \geq x-\pi \right\} \\
&\cup \left\{ (x,y) : \pi \leq x \leq 2\pi, y \leq 2\pi - x, y \geq \pi - x, y \geq x - 2\pi \right\}.
\end{align*}
\end{definition}

\begin{remark}
The set $A_1$ will be instrumental to construct a Lyapunov function. Note that $S_1 \subset A_1$;
therefore, if the conditions of the Lyapunov Theorem hold in $A_1$, they hold in the invariant set $S_1$
 and thus, by equivariance of the vector field, in the entire region $S$.
\label{rem_Lyapfunction}
\end{remark}

\begin{figure}[ptb]
\begin{center}
\includegraphics[
height=4in,
width=4in
]
{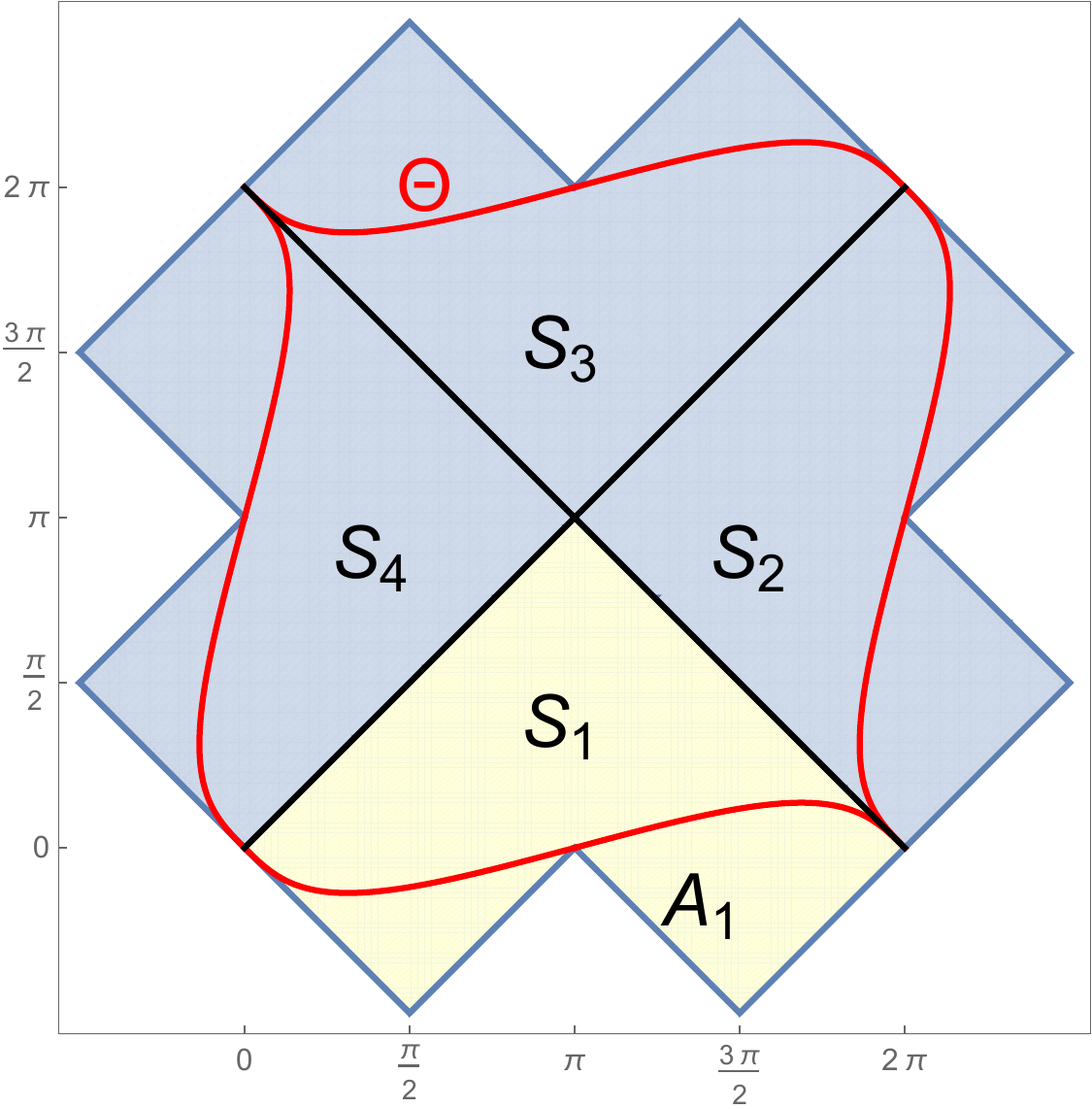}%
\caption{The region where the Lyapunov function is definite positive and the orbital derivative of the Lyapunov function is non-positive is shaded in blue and light yellow, with the light yellow colour representing $A_1$. This region contains the invariant open set $S$, which is enclosed by $\Theta$ (in red). We can observe the intersection of the open domain $S$ with $A_1$, denoted by $S_1$. Our study focuses on this region. The results extend to $S$ by equivariance.
}%
\label{Fig1}%
\end{center}
\end{figure}

\section{Construction of a Lyapunov function}
\label{sec_construction}
Following \cite{bertram1960,giesl2015,lasalle1976stability,sassano2013dynamic}, we now define a Lyapunov function $V\left(
x,y\right)  $ by
\begin{equation}
\label{function_V}
V\left(  x,y\right)  =\left\vert \pi-y\right\vert +\left\vert \pi-x\right\vert
\text{.}
\end{equation}
Note that $V$ is continuous but not $C^1$. However, as stated in definition \ref{def_discrete_Lyap} of section  \ref{sec_intro}, a discrete Lyapunov 
function is only required to be continuous. Moreover, its discrete orbital derivative  \eqref{def_orb_deriv}  is given by
\[
\begin{split} \dot{V}(x, y) = 
    &\left\vert \pi - y - 2a \sin y - a \sin x \right\vert 
    - \left\vert \pi - y \right\vert \\
    &+ \left\vert \pi - x - 2a \sin x - a \sin y \right\vert
    - \left\vert \pi - x \right\vert.
\end{split}
\]
We claim that the orbital derivative $\dot{V}$ is negative in $A_1$. Then, 
by equivariance of the vector field\footnote{The equivariance of the flow is shown in full detail in \cite{BAO2024b}.} we show in paragraph \ref{par:full}) that this implies the asymptotic stability of  
$\left(  \pi,\pi\right)  $ in $S$.

By \eqref{eq:acond}, we consider $0<a<\frac{1}{6}$ throughout the rest of the paper.

\begin{definition}We define the four functions
\[
\begin{split}
\xi_{1}\left(  x,y\right)  =& \pi-y \text{.}\\
\xi_{2}\left(  x,y\right)  =& \pi-y-2a\sin y-a\sin x \text{.}\\
\xi_{3}\left(  x,y\right)  =& \pi-x \text{,}\\
\xi_{4}\left(  x,y\right)  =& \pi-x-2a\sin x-a\sin y \text{,}
\end{split}
\]
\label{def_four functions}
\end{definition}

\noindent With this definition, the orbital derivative is written more compactly as
\begin{equation}
\label{eq_Vdot}
 \dot{V}\left(  x,y\right)
  =\left\vert \xi_{2} \left(  x,y\right) \right\vert -\left\vert \xi_{1} \left(  x,y\right)  \right\vert +\left\vert
\xi_{4} \left(  x,y\right)  \right\vert - \left\vert
\xi_{3} \left(  x,y\right)  \right\vert \text{.}
\end{equation}

\begin{figure}[ptb]
\begin{center}
\includegraphics[height=4in,width=3.95in]{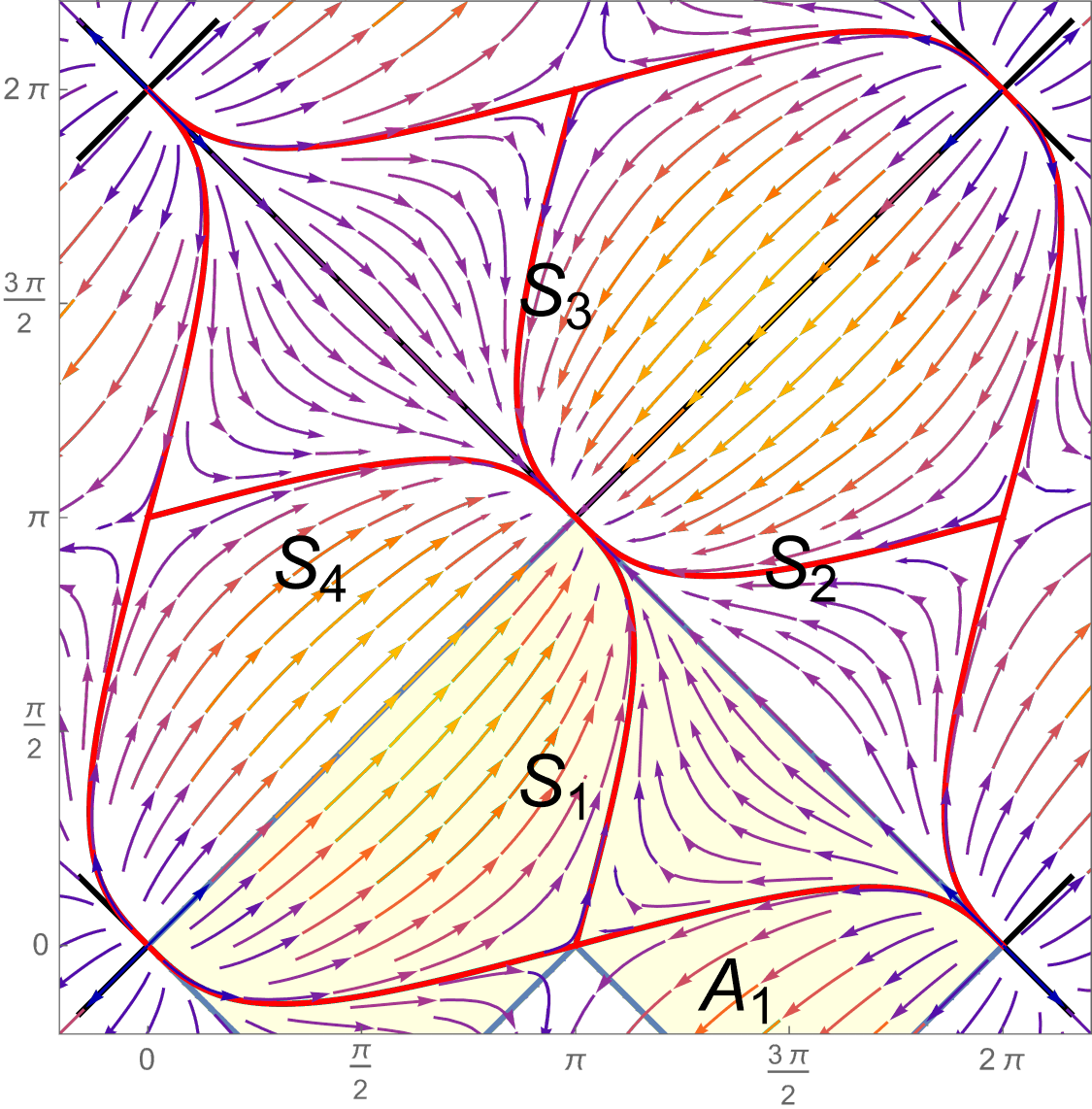}
\end{center}
\caption{Complete phase diagram showing heteroclinic connections between saddle points and stable/unstable foci. The diagram was generated numerically for $a = 0.1$. Our focus is on the yellow shaded region $A_1$ at the bottom of the figure, as the other regions are equivalent by symmetry.
}%
\label{fig:06}%
\end{figure}
\newpage
We claim that the orbital derivative is negative in the interior of $A_1$. We prove this in three steps:
\begin{enumerate}
\item Firstly, we analyse the signs of the individual terms $\xi_i\left(x,t\right)$, $i=1,2,3,4$ occurring in the expression 
\eqref{eq_Vdot} for $\dot{V}$, allowing us to drop the absolute value in the interior of each region of constancy of signs.
\item Secondly, we estimate the sign of the orbital derivative $\dot{V}$ in $A_1$.
\item Thirdly, we extend this analysis to the open set $S$.
\end{enumerate}

\subsection{Analysis of the signs of $\xi_i\left(x,y\right)$, $i=1,2,3,4$} 
We now proceed with the first step of the strategy, analysing the signs of the individual terms 
$\xi_i(x,y)$ in the region $A_1$. We separate the analysis in a sequence of lemmas.
\begin{lemma}\label{lemma1}
In region $A_1$ we have, trivially, that $\xi_{1}\left(x,t\right)=\pi-y\geq0$.
\end{lemma}
Thus, in this region the corresponding absolute value may simply be dropped and 
the orbital derivative is written as
\begin{equation}\label{eq_Vdot2}
 \dot{V}\left(  x,y\right)
  =\left\vert \xi_{2} \left(  x,y\right) \right\vert - \xi_{1} \left(  x,y\right) +\left\vert
\xi_{4} \left(  x,y\right)  \right\vert - \left\vert
\xi_{3} \left(  x,y\right)  \right\vert \text{.}
\end{equation}

\begin{lemma}\label{lemma2}
 We have, 
for all $(x,y) \in  A_1$,
\[
\xi_{2}\left(  x,y\right)  =\pi-y-2a\sin y-a\sin x\geq0\text{.}
\]
\end{lemma}
\begin{proof}
This follows from straightforward computations according to the following strategy.  
The function  $\xi_{2}$ must attain a minimum in the compact set $A_1$. 
We will show that $\xi_{2}$ has no interior extremal points and therefore 
the absolute minimum of $\xi_{2}$ must lie on the boundary of $A_1$. This global minimum 
is then shown
to be zero, implying that $\xi_{2}\left(  x,y\right)\geq 0$ throughout $A_1$.

Consider the possible critical points of $\xi_{2}\left(  x,y\right)  $ in $A_1$. 
The stationarity equations are
\begin{equation*}
\left\{
\begin{aligned}
\partial_{x}\xi_{2}\left(  x,y\right)   & =-a\cos x=0 \\
\partial_{y}\xi_{2}\left(  x,y\right)   & =-1-2a\cos y= 0.
\end{aligned}
\right.
\end{equation*}

%
%
\noindent Since $0<a<\frac{1}{6}$ the last equation has no solutions, implying there are 
no critical points of $\xi_{2}\left(  x,y\right)  $ in the interior of $A_1$. Thus the extrema of 
$\xi_{2}\left(  x,y\right)  $ must lie on the boundary of $A_1$. 

Referring to Fig.~\ref{Fig1}, we see that the boundary of $A_1$ is formed by 
 six line segments, on each of which we now proceed to analyse the behaviour of $\xi_{2}$.

\begin{enumerate}
\item{}\label{item:1} {The top left boundary of $A_1$ is the line segment  
parametrized by $x\in\left[  0,\pi\right]$ and $y=x$. We have 
\[
\left.  \xi_{2}\left(  x,y\right)  \right\vert _{y=x} \equiv g_1\left(  x\right)
=\pi-x-3a\sin x. 
\]
Since 
\[g_1^{\prime}\left(  x\right)
=-1-3a\cos x=0\Leftrightarrow\cos x=-\frac{1}{3a}\]
has no solutions, it follows that there are no local extrema in this set, and the extrema
are in the endpoints of the segment. Now $g_1$ is decreasing along this line segment, with $g_1\left(  0\right)=\pi $, 
and $g_1\left(  \pi\right)=0 $; these are the extrema of $g_1$ on this segment. 
This implies that $g_1\left(  \pi\right)=0 $ is the minimum of $\xi_{2}$ 
on the top left boundary of $A_1$.}

\item{}{The top right boundary segment of $A_1$ is 
parametrized by $x\in\left[ \pi,2\pi\right]$ and $y=2\pi-x$. We have
\[
\left.  \xi_{2}\left(  x,y\right)  \right\vert _{y=2\pi-x} \equiv g_2\left(  x\right)=x-\pi+a\sin x\text{.}
\]
Since 
\[g_2^{\prime}\left(  x\right)
=1+a\cos x=0\Leftrightarrow\cos x=-\frac{1}{3a}\]
has no solutions, it follows that there are no local extrema in this set, and the extrema
are in the endpoints of the segment. 
Thus, $g_2$ is increasing along this line segment, with $g_2\left( \pi\right)=0 $, and $g_2\left(  2\pi\right)=\pi $; these are the extrema of $g_2$ on this segment. 
This implies that $g_2\left(  \pi\right)=0 $ is the minimum value of $\xi_{2}$ on the top right boundary segment of $A_1$, assumed at the point $(\pi,\pi)$.
}

\item{}{
The bottom far left 
boundary of $A_1$ is the line segment  
parametrized by $x\in\left[  0,\frac{\pi}{2}\right]$ and $y=-x$. We have 
\[
\left.  \xi_{2}\left(  x,y\right)  \right\vert _{y=-x}=\pi+x+a\sin x>0\text{,
} 
\]
and therefore $\xi_{2}\left(  x,y\right)$ is trivially positive in this line segment.
}

\item{}{The adjacent segment (middle left) in the bottom boundary of $A_1$ is parametrized by
$x\in\left[  \frac{\pi}{2},\pi\right]$ and $y=x-\pi$.
\[
\left.  \xi_{2}\left(  x,y\right)  \right\vert _{y=x-\pi}=2\pi-x+a\sin x>0\text{,
}
\]
from which we conclude immediately that  $\xi_{2}\left(  x,y\right)$ is positive in this line segment.}
 
\item{}{
The adjacent segment (middle right) in the bottom boundary  $A_1$  is parametrized by $x\in\left[ \pi,\frac{3\pi}{2} \right]$ and $y=\pi-x$.
\[
\left.  \xi_{2}\left(  x,y\right)  \right\vert _{y=\pi-x}=x-3a\sin x>0\text{,
} 
\]
from which it follows that $\xi_{2}\left(  x,y\right)$ is trivially positive in this line segment. 
}
\item{}
{The last segment on the bottom boundary of $A_1$, on the far right, is parametrized by
 $x\in\left[ \frac{3\pi}{2},2\pi\right]$  and $y=x-2\pi$.  in the bottom of $A_1$. On this segment we have
\[
\left.  \xi_{2}\left(  x,y\right)  \right\vert _{y=x-2\pi}=3\pi-x-3a\sin x>0\text{,
} 
\]
from which it follows that  $\xi_{2}\left(  x,y\right)$ is trivially positive on this line segment. We may also conclude that the maximum of $\xi_{2}\left(  x,y\right)$ is attained at $\left( \frac{3\pi}{2},-\frac{\pi}{2}\right)$, and its value is $\frac{3\pi}{2}+3a$.
}
\end{enumerate}

We conclude that the absolute minimum of  $\xi_{2}\left(  x,y\right)$ in $A_1$ is zero and is 
attained only at $\left( \pi,\pi\right)$.
\end{proof}

\vspace{2mm} 
Summing up, Lemma~\ref{lemma2} states that 
\[
\xi_{2}\left(  x,y\right) \geq0 \text{ in } A_1,
\]
with $\xi_{2}\left(  x,y\right) = 0$ if and only if $\left(  x,y\right)  =\left(  \pi,\pi\right)  $. 
Recalling expression \eqref{eq_Vdot2} for the orbital derivative together with Lemma~\ref{lemma2}, we have, for $(x,y) \in A_1$,
\[
\begin{split} \dot{V}\left(  x,y\right)   & =\xi_{2}\left(  x,y\right)-\xi_{1}\left(  x,y\right)+\left\vert \xi_{4}\left(  x,y\right) \right\vert -\left\vert \xi_{3}\left(  x,y\right) \right\vert \\
 & =-2a\sin y-a\sin x+\left\vert \xi_{4}\left(  x,y\right) \right\vert -\left\vert \xi_{3}\left(  x,y\right) \right\vert .
\end{split}
\]

\begin{figure}
[ptb]
\includegraphics[
height=3in,
width=4in,center
]%
{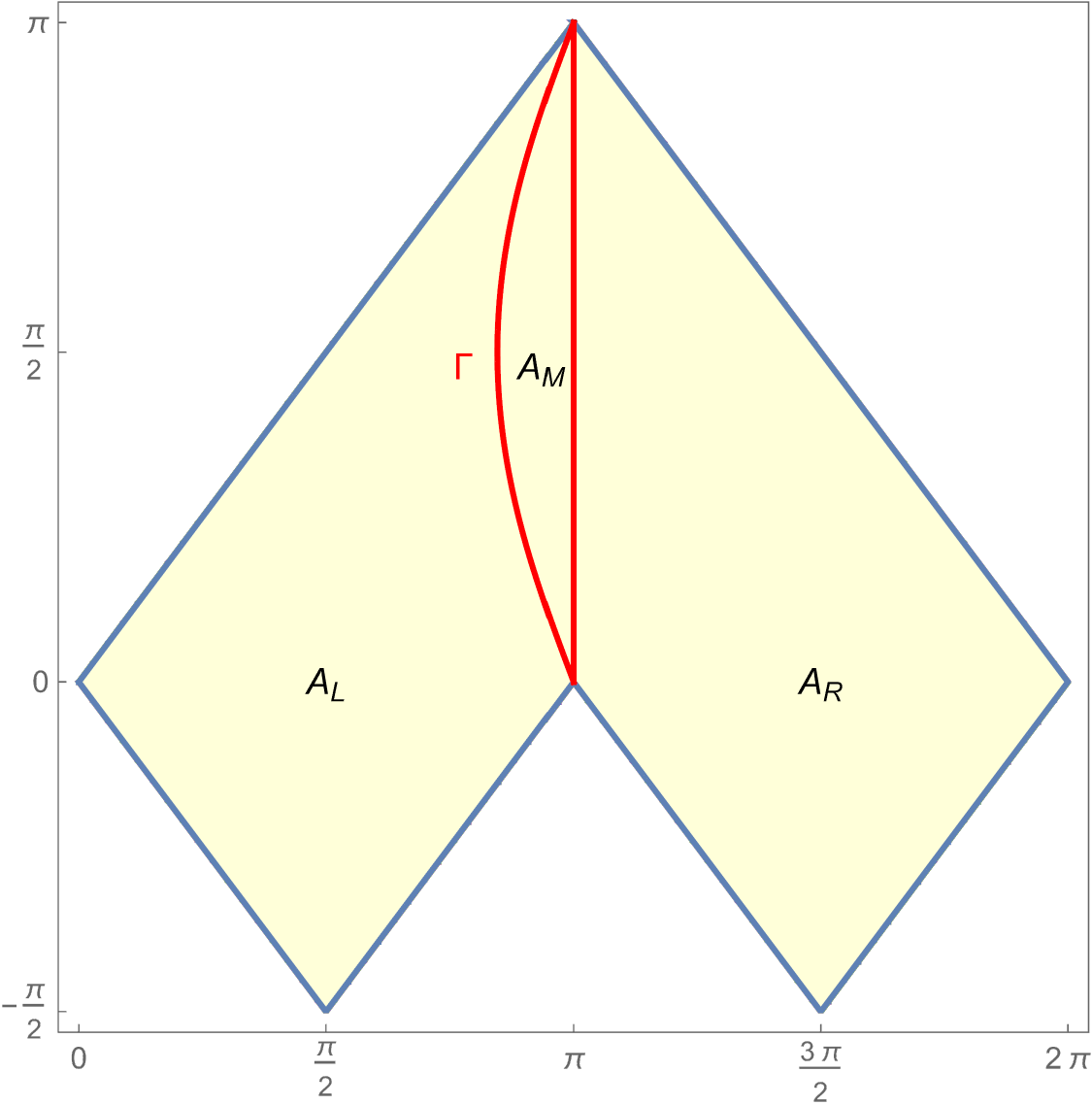}%
\caption{Detail of the region $A_1$ and its subdivisions, where the three possible combinations of
signs of the absolute values of $\xi_{3} $ and $\xi_{4}$ in the orbital derivative can occur.}%
\label{Fig2}%
\end{figure}
For the analysis of the signs of $\xi_{3}$ and $\xi_{4}$, it becomes necessary to further split $A_1$  into subregions.

\begin{definition}\label{def:regions} 
Define the two following subregions $A^-\, $,  $A^+$ of $A_1$: 
 \begin{align*}
  A^- &= \{ (x,y) \in A_1 : 0 \leq x \leq \pi \}, \\
 A^+   &= \{ (x,y) \in A_1 : \pi \leq x \leq 2\pi \}.
\end{align*}
\end{definition}

\begin{lemma}\label{lemma3}
The function $\xi_{3}=\pi-x$ is non-negative in $A^-$ and non-positive in $A^+$.
\end{lemma}
\proof This is a trivial check: we have%
\[
-\left\vert \xi_{3}\right\vert =-\pi+x   \leq 0 \ \text{ for } (x,y) \in A^-  
\]
and
\[
-\left\vert \xi_{3}\right\vert =\pi-x \geq 0 \     \text{ for } (x,y) \in A^+  \text{.}
\]
\qed 

\begin{remark}\label{rem:2}
It is convenient to write the expression of the orbital derivative $\dot{V}$
in each of the two subregions $A^-$ and $A^+$:

\begin{equation*}
 \dot{V}\left(  x,y\right) = \left\{
\begin{aligned}
-2a\sin y-a\sin x-\pi+x+\left\vert \xi_{4}\left(  x,y\right) \right\vert, 
 & \ \ \ (x,y) \in A^- \text{,} \\
-2a\sin y-a\sin x+\pi-x+\left\vert \xi_{4}\left(  x,y\right) \right\vert, & \ \ \ (x,y) \in A^+ \text{.}
\end{aligned}
\right.
\end{equation*}
\end{remark}

To complete the analysis of the signs of the terms inside the absolute values in the orbital derivative, we now tackle  the remaining term 
$\xi_{4}\left(  x,y\right)  =\pi-x-2a\sin x-a\sin y$. Here the analysis is more subtle; in order to perform it it is necessary to define the regions where $\xi_{4}\left(  x,y\right)$ does not change sign.

\begin{lemma}\label{lemma:implicit}
The equation 
\begin{equation}\label{eq:implicit}
\xi_{4}\left(  x,y\right)=0
\end{equation}
defines implicitly a unique analytic function $x=\gamma \left( y \right)$ such that the curve 
$\Gamma$ parametrized by $(\gamma(y), y), \ 0 \leq y \leq \pi$, is an analytic curve 
connecting the fixed points $\left( \pi,0 \right)$ and $\left( \pi,\pi \right)$ of the system \eqref{Eq3}.
\end{lemma}

\begin{proof}
Fixing a value of $y=y_0 \in \left[0,\pi\right]$, we consider the left hand side of \eqref{eq:implicit} as a function of $x$: 

\[
H_{y_0}\left( x \right)=\pi-x-2a\sin x-a\sin y_0.
\]
We then have
\begin{equation}
\frac{dH_{y_0}\left(  x\right) }{dx}\ =-1-2a\cos
x<0 
\label{eq_derivative0}
\end{equation}
for all $y_0 \in \left[  0,\pi\right]$, showing that $H_{y_0}   \left( x \right)$ is a strictly decreasing 
function of $x$. Moreover, $\lim_{x\rightarrow \pm \infty} H_{y_0}   \left( x \right)=\mp \infty$. Therefore, it follows from the Intermediate Value Theorem that the equation 
$H_{y_0}(x) = 0$ has, for each $y_0$, a unique solution, which we denote by $x_0 = \gamma(y_0)$.
From this fact and \eqref{eq_derivative0}, we may apply the analytic Implicit Function Theorem locally at the
point $(x_0, y_0)$. The fact that both conditions hold for all $y_0 \in \left[0,\pi\right]$ implies that 
these local functions join smoothly to define a unique analytic function 
$x=\gamma \left( y \right), \ 0 \leq y \leq \pi$. Thus the curve $\Gamma$ parametrized by 
$(\gamma(y), y), \ 0 \leq y \leq \pi$ is an analytic curve 
connecting the fixed points $\left( \pi,0 \right)$ and $\left( \pi,\pi \right)$, as stated. 
\end{proof}

\begin{remark} \label{rem:gamma} 
Analysis of the function $x=\gamma\left( y \right)$ shows that it admits a minimum at 
$y=\frac{\pi}{2}$. This minimum is less than $\pi$ since $x=\gamma\left( y \right)$ is decreasing with $y$.
This implies that the curve $\Gamma$ 
lies on the region $A^-$. 
\end{remark}

Remark~\ref{rem:gamma} allows us to decompose the region $A^-$ into the 
following regions of interest from the point of view of constancy of sign of $\xi_4$.

\begin{definition}\label{def:regionsLR} 
 
Define the two following subregions $A_L$ and  $A_M$ of $A_1$: 
 
\begin{align*}
& A_L=\left\{ \left( x,y \right) \in A_1:
x \leq \gamma \left( y \right) \right\}, 
\\
& A_M=\left\{ \left( x,y \right) \in A_1:
 \gamma \left( y \right)\leq x \leq \pi \right\}.
\end{align*}
To maintain the consistency of notation we rename $A^+$ to $A_R$.
\end{definition}

\begin{remark}
\label{rem:1}
The two subregions $A_L$ and $A_M$ are in fact subsets of $A^-$; indeed $A^- = A_L \cup A_M$, and furthermore
$A_L$ and $  A_M$ intersect along the curve $\Gamma$.
Thus $A_1= A_L\cup A_M \cup A_R$. See Fig.~\ref{Fig2}.
\end{remark}

The next result completes the analysis of the signs of $\xi_{4}$. 

\begin{lemma}
\label{lemma5}
We have:
\begin{enumerate}
\item $\xi_{4}(x,y) \geq 0$ in  $A_L$;
\item $\xi_{4}(x,y) \leq 0$ in  $A_M \cup A_R$.
\end{enumerate}
\end{lemma}

\begin{proof} We begin by noting that,  by a similar reasoning to  Lemma~\ref{lemma2}, $\xi_{4}\left(  x,y\right)  $ has no
critical points inside $A_1$. The problem thus reduces to the analysis of $\xi_{4}$ on the 
boundaries of the relevant regions, which now are $A_L$ and $A_M \cup A_R$. 

The regions $A_L$ and $A_M \cup A_R$ have as common boundary the curve $\Gamma$
constructed in Lemma~\ref{lemma:implicit}. Note that we already know that $\xi_{4} \equiv 0$ 
along $\Gamma$, and consequently we only have to consider the three boundary line segments 
for each region. 

We now proceed to prove the first statement.

\begin{enumerate}

\item The top  boundary of $A_L$ is the line segment  
parametrized by $x\in\left[  0,\pi\right]$ and $y=x$. We have, for all $x\in\left[  0,\pi\right] $, 
\[
\left.  \xi_{4}\left(  x,y\right)  \right\vert _{y=x}=\pi-x-3a\sin
x\geq0\text{. }
\]
This is exactly the same situation of Lemma~\ref{lemma2} (\ref{item:1}). The analogous calculations
lead to the conclusion that the function $\xi_{4}$ is non-negative on this segment,
 the only zero being at $(\pi,\pi)$.\\ 

\item{} The bottom left boundary segment of $A_L$ is 
parametrized by $x\in\left[  0,\frac{\pi}{2}\right]$ and $y=-x$. We then trivially have
\[
\left.  \xi_{4}\left(  x,y\right)  \right\vert _{y=-x}=\pi-x-a\sin x>0\text{.
} 
\]

\item{} The bottom right boundary segment of $A_L$ is 
parametrized by $x\in\left[  \frac{\pi}{2},\pi\right] $ and $y=x-\pi$. We then have
\[
\left.  \xi_{4}\left(  x,y\right)  \right\vert _{y=x-\pi}=\pi-x-a\sin
x\geq0\text{, } 
\]
with equality if and only if $x=\pi$.
\end{enumerate}

From (1)-(3) it follows that in $A_L$ there exists a line of minimizers along the curve 
$\Gamma$ with absolute minimum zero. This implies 
\[
\xi_{4}\left( x,y \right)\geq 0 \text{ in } A_L, \] 
concluding the proof of the first statement.

To prove the second statement, we concentrate on the set $A_M\cup A_R$.

\begin{enumerate}
 \item The top  boundary of $A_M\cup A_R$ is the line segment  
parametrized by $x\in\left[ \pi,2\pi \right] $ and $y=2\pi-x$. We then have
\[
\left.  \xi_{4}\left(  x,y\right)  \right\vert _{y=2\pi-x}=\pi-x-a\sin
x\leq0\text{, } 
\]
the derivative is $-1-a \cos x$. Thus, the function $\xi_{4}$ is non-positive on this segment with a unique zero at $(\pi,\pi)$.\\

\item{} The bottom left boundary segment of $A_M\cup A_R$ is   
parametrized by
 $x\in\left[  \pi,\frac{3\pi}{2}\right]$ and $y=\pi-x$. We have 
\[
\left.  \xi_{4}\left(  x,y\right)  \right\vert _{y=\pi-x}=\pi-x-3a\sin x\leq 0\text{,
} 
\]
with a unique zero at $(\pi,0)$ and again decreasing with $x$.
\item{} The bottom left boundary segment of $A_M\cup A_R$ is   
parametrized by
 $x\in\left[  \frac{3\pi}{2},2\pi\right] $ and $y=x-2\pi$. We have trivially
\[
\left.  \xi_{4}\left(  x,y\right)  \right\vert _{y=x-2\pi}=\pi-x-3a\sin
x<0.
\]
\end{enumerate}
We conclude that in $A_M\cup A_R$ there exists a line of maximizers of $\xi_{4}$ along the
curve $\Gamma$ with the maximum value zero. This concludes the proof of the second statement,
finishing the proof of the lemma.

\end{proof}

\subsection{Sign of the orbital derivative}
\begin{theorem}
\label{thm_sign_A1}
The orbital derivative is non-positive in \( A_1 \). Specifically, it is zero exactly at the lower edges of \( A_1 \) and at the fixed point \( (\pi, \pi) \), while it is negative at all other points in \( A_1 \).
\end{theorem}
\begin{proof}

To compute the orbital derivative, we recall Remark~\ref{rem:2}. We substitute the findings of Lemma~\ref{lemma5} into the statements of the mentioned remark. Now, we consider three cases corresponding to the regions \( A_L \), \( A_M \), and \( A_R \).
\begin{enumerate}
\item{} We consider the region $A_L$, the orbital derivative is
\[
\begin{split} \dot{V}\left(  x,y\right)    & =-2a\sin y-a\sin x-\pi+x+ \xi_{4}\left(  x,y\right)\\
& =-3a\sin y-3a\sin x \text{;}
\end{split}
\]
\item{} in the region $A_M$, the orbital derivative is
\[
\begin{split} \dot{V}\left(  x,y\right)    & =-2a\sin y-a\sin x-\pi+x- \xi_{4}\left(  x,y\right)\\
&=2x-2\pi-a\sin y+a\sin x
 \text{;}
\end{split}
\]
\item{} in the region $A_R$, the orbital derivative is
\[
\begin{split} \dot{V}\left(  x,y\right)    & =-2a\sin y-a\sin x+\pi-x- \xi_{4}\left(  x,y\right) \\
& =-a\sin y+a\sin x \text{.}
\end{split}
\]
\end{enumerate}
\begin{enumerate}
\item{} Let us consider the region
$A_{L}$. Here we have
\[
\dot{V}\left(  x,y\right)  =-3a\sin y-3a\sin x.
\]
There are no local extrema in the interior of $A_L$ since the stationarity system
\[
\begin{split}
\partial_x \dot{V}\left(  x,y\right)=-3a\cos x = 0\\
\partial_y \dot{V}\left(  x,y\right)=-3a\cos y = 0
\end{split}
\]
has  the only solutions  $(\frac{\pi}{2},-\frac{\pi}{2})$ and $(\frac{\pi}{2},\frac{\pi}{2})$, which lie in the boundary of $A_L$. 
The values of $\dot{V}$ at those points are
\[
\begin{split}
& \dot{V}\left( \frac{\pi}{2},\frac{\pi}{2}\right)=-6a \\
& \dot{V}\left( \frac{\pi}{2},-\frac{\pi}{2}\right)=0.
\end{split}
\]
By inspection we see that the orbital derivative is zero in
$\left(  \pi,\pi\right)  $.
We assess now the boundary of $A_L$, formed by three line segments and  $\Gamma$.
\begin{enumerate}
\item {} On the lower left edge of $A_L$, $x\in \left[ 0, \frac{\pi}{2} \right] $ and $y=-x$, we have
\begin{align*}
& \left.  \dot{V}\left(  x,y\right)  \right\vert _{y=-x}
 =3a\sin x-3a\sin x=0.
\end{align*}
\item {} On the lower right edge of $A_L$,  $x\in \left[  \frac{\pi}{2} ,\pi \right]$ and $y=x-\pi$, we have
\begin{align*}
& \left.  \dot{V}\left(  x,y\right)  \right\vert _{y=x-\pi}=3a\sin x-3a\sin x=0.
\end{align*}
\item {} On the top left edge of $A_L$, $x\in \left]  0 ,\pi \right[$ and $y=x$, $\dot{V}$ is negative: 
\begin{align*}
& \left.  \dot{V}\left(  x,y\right)  \right\vert _{y=x}=-6a\sin x<0.
\end{align*}
\item {} On $\Gamma$, $\dot{V}$ is negative.
Since $0<\gamma(y)<\pi$ 
for $y\in \left] 0,\pi \right[$, as seen in Remark~\ref{rem:gamma}, we have
\begin{align*}
& \left.  \dot{V}\left(  x,y\right)  \right\vert _{x=\gamma\left(  y\right)} =-3a\sin \gamma (y)-3a\sin y< 0.
\end{align*}
\end{enumerate}
In conclusion, the maximum value of the orbital derivative in the compact set $A_L$ is $0$, and the minimum is $-6a$.

\item{} Consider the region $A_{M}$. We have%
\begin{align*}
\dot{V}\left(  x,y\right)
& =-2\left(  \pi-x\right)  +a\sin x-a\sin y,
\end{align*}
and again there are no local extrema in the interior of $A_M$, since $\partial_x \dot{V}=2+a \cos x=0$ 
has no solutions.
By continuity, 
\[ \dot{V}((\pi,0) ) = \dot{V}((\pi,\pi) ) =0,  \] 
while, as seen above, $\dot{V}$ is negative on $\Gamma$, which is shared with  $A_L$.
It only remains to consider the right edge of $A_M$, where $x=\pi$ for $y\in \left] 0,\pi \right[$. 
The orbital derivative over that edge is
\begin{align*}
& \left.  \dot{V}\left(  x,y\right)  \right\vert _{x=\pi} =-a\sin y<0,
\end{align*}
which is again negative except at the points $\left(  \pi,0\right)  \,$\ and $\left(  \pi,\pi\right)$ where it attains its maximum value   $0$.
\item{} Consider the region $A_{R}$. We have%
\begin{align*}
\dot{V}\left(  x,y\right)  
& =-a\sin y+a\sin x.
\end{align*}
There are no local extrema in the interior of $A_R$ since the stationarity system
\[
\begin{split}
\partial_x \dot{V}\left(  x,y\right)=a\cos x = 0\\
\partial_y \dot{V}\left(  x,y\right)=-a\cos y = 0
\end{split}
\]
has the only solutions  $(\frac{3\pi}{2},-\frac{\pi}{2})$ and $(\frac{3\pi}{2},\frac{\pi}{2})$, which 
lie in the boundary of $A_R$. The values of $\dot{V}$ at those points are
\[
\begin{split}
& \dot{V}\left( \frac{3\pi}{2},\frac{\pi}{2}\right)=-2a \\
& \dot{V}\left( \frac{3\pi}{2},-\frac{\pi}{2}\right)=0.
\end{split}
\]

We now assess the boundary of $A_R$, formed by three line segments and the left vertical edge shared with $A_M$. At the lower edges of  $A_{R}$ the orbital derivative is again $0$, as we can see in the next four items.%
\begin{enumerate}
\item {} {In the left lower edge of $A_R$, $x\in \left[ \pi, \frac{3\pi}{2} \right] $ and $y-\pi-x$ we have 
\begin{align*}
& \left.  \dot{V}\left(  x,y\right)  \right\vert _{y=\pi-x} =-a\sin\left(  \pi-x\right)  +a\sin x=0.
\end{align*}}
\item{} {In the right lower edge of $A_R$, $x\in \left[ \frac{3\pi}{2},2\pi  \right] $ and $y=x-2\pi$ we have 
\begin{align*}
& \left.  \dot{V}\left(  x,y\right)  \right\vert _{y=x-2\pi} =-a\sin\left(  x-2\pi\right)  +a\sin x=0.
\end{align*}}
\item{} {In the right top edge of $A_R$ the orbital derivative is trivially negative for $x\in \left] \pi, 2\pi \right[ $ and $y=2\pi-x$ 
\begin{align*}
& \left.  \dot{V}\left(  x,y\right)  \right\vert _{y=2\pi-x}  =2a\sin x <0.
\end{align*}
}
\item{} {In the vertical left edge of $A_R$ shared with $A_M$ the orbital derivative is trivially negative by continuity for $y\in \left] 0, \pi \right[ $.
}
\end{enumerate}
\end{enumerate}

Combining all the cases, the orbital derivative is negative in the interior of
$A_1$. The orbital derivative attains its maximum $0$ at the fixed point $(\pi,\pi)$ and at the lower
edges of $A_1$, which contain the other fixed points of $F$ in $A_1$. 
\end{proof}
\begin{remark} Notice 
that the lower edges of $A_1$ only 
intersect the heteroclinics at the fixed points of the vector field, 
as depicted in Fig.~\ref{Fig1}.
\end{remark}

\subsection{Full picture}\label{par:full}

\begin{figure}
[ptb]
\begin{center}
\includegraphics[
height=4in,
width=4in
]
{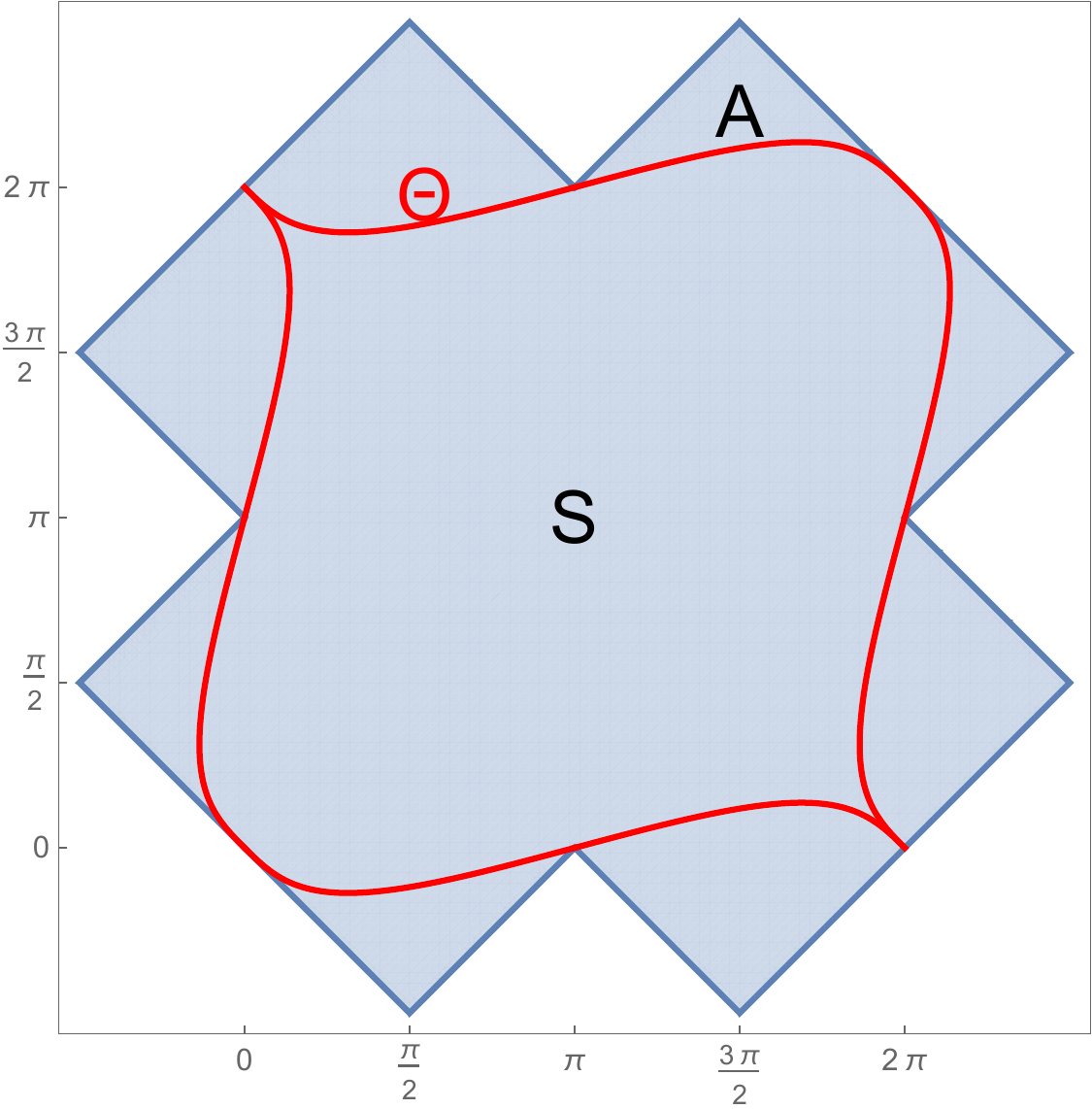}%
\caption{The compact set $A$ containing the open set $S$ where the Lyapunov function has negative orbital derivative -- except at the asymptotically stable node -- is shaded in blue. We see as well the set $\Theta=\partial S$.
}%
\label{Figfull}%
\end{center}
\end{figure}

The preceding results 
in this section, namely Lemmas~\ref{lemma1} through \ref{lemma5} and Theorem~\ref{thm_sign_A1}, show that the function $V$ defined globally by \eqref{function_V} acts as a Lyapunov function
on the set $A_1$, 
since it is positive definite and $\dot{V}$ is non-positive along orbits.
 We will now show that these properties actually extend by symmetry to the whole open set $S$ containing the fixed 
point $(\pi,\pi)$.

\begin{definition}
We define the compact set $A_2$, containing the domain $S_2$
\begin{align*}
A_2 &= \left\{ (x,y) : \pi \leq y \leq 2\pi, y \leq x, y \leq 4\pi-x,y \geq x-\pi \right\} \\
&\cup \left\{ (x,y) : 0 \leq y \leq \pi, y \geq 2\pi - x, y \leq 3\pi - x, y \geq x - 2\pi \right\}.
\end{align*}
\end{definition}
The set $A_2$ is the reflection of the set $A_1$ along the line $y=2\pi-x$.

\begin{definition}
We define the compact set $A_3$, containing the domain $S_3$
\begin{align*}
A_3 &= \left\{ (x,y) : \pi \leq x \leq 2\pi, y \geq x, y \leq 4\pi-x,y \leq x+\pi \right\} \\
&\cup \left\{ (x,y) : 0 \leq x \leq \pi, y \geq 2\pi - x, y \leq 3\pi - x, y \leq x + 2\pi \right\}.
\end{align*}
\end{definition}
The set $A_3$ is the rotation of the set $A_1$ around $\left( \pi, \pi \right)$.
\begin{definition}
We define the compact set $A_4$, containing the domain $S_4$
\begin{align*}
A_4 &= \left\{ (x,y) : 0 \leq y \leq \pi, y \geq x, y \leq -x,y \leq x+\pi \right\} \\
&\cup \left\{ (x,y) : \pi \leq x \leq 2\pi, y \leq 2\pi - x, y \geq \pi - x, y \leq x + 2\pi \right\}.
\end{align*}
\end{definition}
The set $A_4$ is the reflection of the set $A_1$ along the line $y=x$.

\begin{figure}
[ptb]
\begin{center}
\includegraphics[height=2in,width=2in]{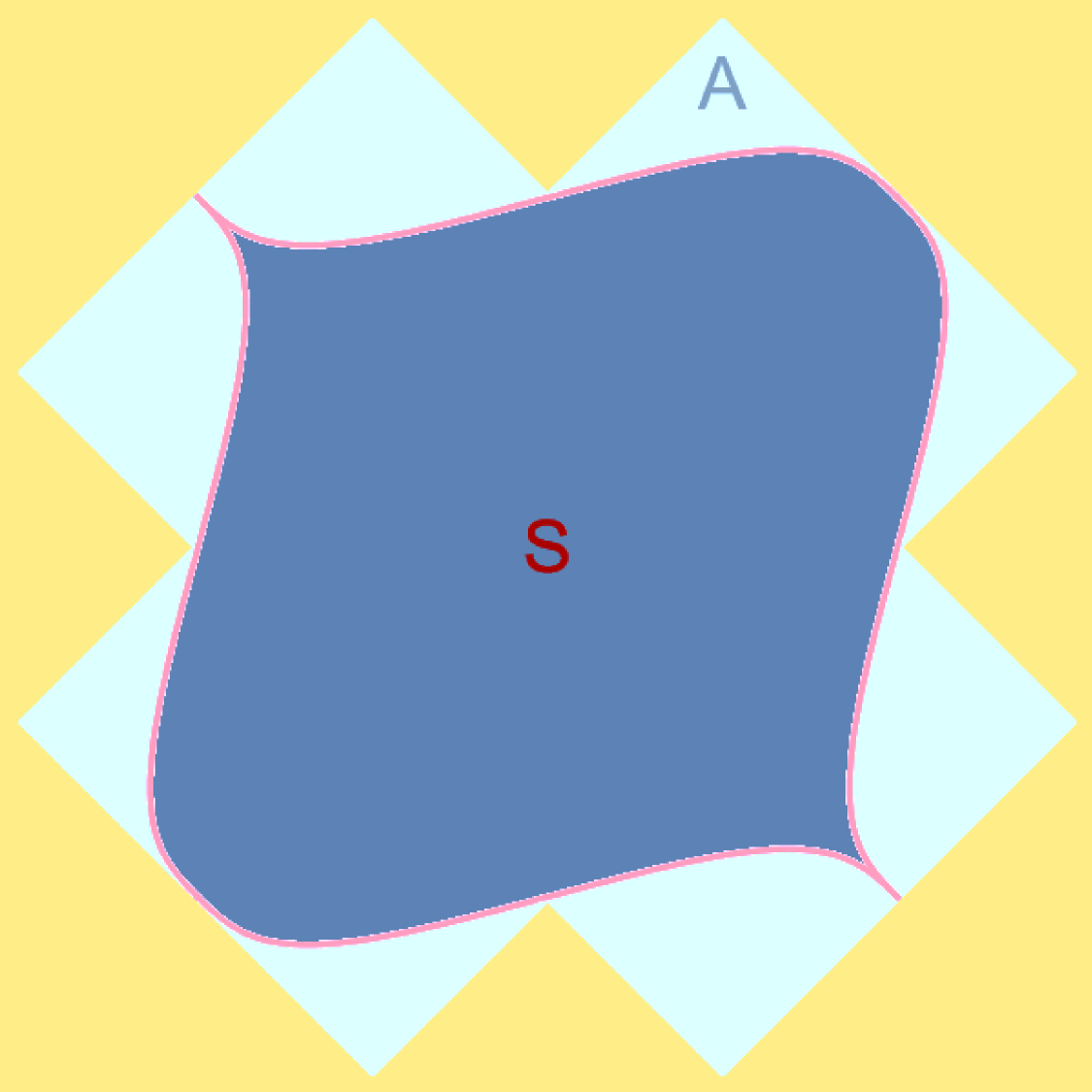}
\caption{The compact set $A$, in white, contains the open set $S$ in dark blue, where the Lyapunov function is positive and its orbital derivative $\dot{V}$ is negative, except at the asymptotically stable node $\left( \pi,\pi \right) $.}
\label{Figset}
\end{center}
\end{figure}

The union
\[
A=\bigcup_{i=1}^4 A_i
\]
is a compact set such that $S\subset A$.

\begin{remark}Recalling the construction of the involutions $\Phi_j$ from 
Definition \ref{def:phii} of $\Phi_j$,  it is immediate to check that, similarly to \eqref{eq:phi_relation},
\begin{align}\label{eq:Ai_relation} 
\Phi_j\left( A_1 \right) = A_{j} \text{ for } j = 1, 2, 3,4.
\end{align}
\end{remark}

\noindent We now show that $V$ is invariant relative to $\Phi_j$,  
 $j=1,2,3,4$.
\begin{lemma}\label{lemma:eqV}
We have
\begin{equation}
V\left( \Phi_j \left( x,y \right) \right)=V\left(  x,y  \right)
\end{equation}
\end{lemma}
\begin{proof}
This is a simple verification.
Recall the construction of the maps $\Phi_j$ from Definition \ref{def:phii}. 
The map $\Phi_1$ is the identity, so the statement is trivial. For the other maps we have
\begin{align*}
& V\left( \Phi_2 \left( x,y \right) \right)=V\left( 2\pi-y,2\pi-x \right)=
\left\vert y-\pi\right\vert +\left\vert x-\pi\right\vert =
V\left(  x,y  \right),\\
&  V\left( \Phi_3 \left( x,y \right) \right)=V\left( 2\pi-x,2\pi-y \right)=
\left\vert x-\pi\right\vert +\left\vert y-\pi\right\vert =
V\left(  x,y  \right),\\
&  V\left( \Phi_4 \left( x,y \right) \right)=V\left( y,x \right)=
\left\vert \pi-y\right\vert +\left\vert \pi-x\right\vert =
V\left(  x,y  \right).
\end{align*}
\end{proof}

\begin{theorem}\label{thm:eqOD}
The orbital derivative $\dot{V}$ is invariant under the action of the $\Phi_j$, $j=1 \ldots 4$:
\begin{equation}
\dot{V}\left( \Phi_j \left( x,y \right) \right)=\dot{V}\left(  x,y  \right).
\end{equation}

\end{theorem}
\begin{proof}
For each \( \Phi_j \), \( j = 1, \dots, 4 \), we have
\begin{align*}
\dot{V}\left( \Phi_j \left( x, y \right) \right) 
&= V\left( F \circ \Phi_j \left( x, y \right) \right) - V\left( \Phi_j \left( x, y \right) \right),
\end{align*}
from which, using the commutativity of \( F \) and \( \Phi_j \) as noted in Remark~\ref{rem:simple_exercise}, we obtain

\begin{align*}
& \dot{V}\left( \Phi_j \left( x,y \right) \right)= V\left( \Phi_j\circ F \left( x,y \right) \right)-V\left( \Phi_j \left( x,y \right) \right)
\end{align*}
and, applying now Lemma~\ref{lemma:eqV}, it follows that
\begin{align*}
& \dot{V}\left( \Phi_j \left( x,y \right) \right)= V\left(  F \left( x,y \right) \right)-V\left( \left( x,y \right) \right)=\dot{V} \left( x,y  \right)
\end{align*}
\end{proof}

Theorem~\ref{thm:eqOD} implies that 
 the orbital derivative $\dot{V}$ is non-positive in $A_1$ if and only if it is 
non-positive in $A_2$, $A_3$ and $A_4$. Moreover, $\dot{V}$ is in fact negative in the interior of 
\[
A=\bigcup_{j=1}^4A_j, 
\]
except at $\left( \pi,\pi \right)$, since on the boundaries of $A_1$ along the lines $y=x$ and $y=2\pi-x$ it is negative 
(except at its endpoints which coincide with fixed points of $F$). 
Since $S\subset \text{Int}\left( A \right)$, the orbital derivative is negative $\dot{V}$ throughout the whole open set  
$S$ except at the fixed point $\left( \pi,\pi \right)$, where $\dot{V}\left( \pi,\pi \right)=0$.


Recalling the discrete Lyapunov Theorem~\ref{Lyapunov for maps}, we may now summarize our
results in the following form. 

\begin{theorem}
\label{Lyapunov_final_result}
The function \( V(x,y) \) defined in \eqref{function_V} is a strict Lyapunov function for the fixed point
\( \left( \pi, \pi \right) \) in the open set \( S \). Consequently, \( \left\{ \left( \pi, \pi \right) \right\} \)
is asymptotically stable with an open basin of attraction \( S \), i.e., it is Lyapunov stable and the \( \omega \)-limit set of all
initial conditions \( \left( x_0, y_0 \right) \in S \).

\end{theorem}


\section{Conclusion}

Lyapunov functions, introduced over a century ago by Lyapunov  \cite{lyapunov1892}, remain an essential tool for analysing the stability of dynamical systems, in both theoretical and practical contexts, across science and engineering. Constructing these functions is an ongoing challenge that impacts various fields, from real-world engineering applications to  mathematics proper. The discovery of a dynamical system admitting an explicit Lyapunov function may  thus be considered a striking situation.

The diffeomorphism \eqref{Eq3}, arising in the problem of Synchronisation of three limit cycle oscillators in a line with nearest-neighbour interactions, was studied in \cite{BAO2024b} and shown to have a unique ($\pi, \pi)$ fixed point asymptotically stable attractor. This was done by constructing a network of heteroclinic connections and showing laboriously that the fixed point is asymptotically stable and that its basin of attraction is the interior of the region bounded by the heteroclinics.

In this paper we prove asymptotic stability of the fixed point by constructing a discrete Lyapunov function. This construction 
is, of course, deeply inspired by the underlying geometry of the phase space symmetries and dynamics. 
It is also crucially linked to the fact that discrete Lyapunov functions are only required to be continuous. 

\subsection*{Acknowledgements}

The author Jorge Buescu was partially supported by  Fundação para a Ciência e
Tecnologia, Fundação para a Ciência e Tecnologia, UIDB/04561/2020.

The author Emma D'Aniello was partially supported by the program Erasmus+, the projects 2024 DYNAMIChE of the INdAM Group GNAMPA, PRIN 2022 QNT4GREEN, DAISY of the Univ. Vanvitelli (D.R. 111, 09/02/2024), and the group UMI-TAA: Approximation Theory and Applications.

The author Henrique M. Oliveira was funded by FCT/Portugal through project UIDB/04459/2020 with DOI identifier 10-54499/UIDP/04459/2020.

\subsection*{Data availability}
Not applicable. The proofs and calculations were presented in the current article. Any queries can be addressed to the corresponding author, Henrique M. Oliveira.

\subsection*{Disclosure of interest}
The authors report no conflict of interest.

\bibliographystyle{abbrv}
\bibliography{BibloH2}

\end{document}